\documentclass{daj}

\usepackage{amsmath,amsthm,amssymb,mathrsfs,mathtools,stmaryrd,color}

\usepackage{float}
\usepackage{graphicx}

\usepackage[utf8]{inputenc}
\usepackage[T1]{fontenc}

\usepackage{enumitem}
\setlist[enumerate]{itemsep=2pt,parsep=2pt,before={\parskip=2pt}}

\usepackage[capitalize]{cleveref}

\usepackage{tikz}
\usetikzlibrary{calc}
\usetikzlibrary{math}
\usetikzlibrary{arrows.meta}

\newtheorem{theorem}{Theorem}[section]
\newtheorem*{theorem*}{Theorem}
\newtheorem*{definition*}{Definition}
\newtheorem{proposition}[theorem]{Proposition}
\newtheorem{lemma}[theorem]{Lemma}
\newtheorem{corollary}[theorem]{Corollary}

\theoremstyle{definition}
\newtheorem{definition}[theorem]{Definition}

\newtheorem{remark}[theorem]{Remark}
\newtheorem{example}[theorem]{Example}
\newtheorem{notation}[theorem]{Notation}

\newcommand{\R}{\ensuremath{\mathbf{R}}}
\newcommand{\C}{\ensuremath{\mathbf{C}}}
\newcommand{\Z}{\ensuremath{\mathbf{Z}}}
\newcommand{\N}{\ensuremath{\mathbf{N}}}
\newcommand{\F}{\ensuremath{\mathbf{F}}}
\newcommand{\eps}{\varepsilon}

\newcommand{\m}[1]{\mu_{#1}}
\newcommand{\mh}[1]{\widehat{\mu_{#1}}}
\newcommand{\ind}[1]{\mathbf{1}_{#1}}
\newcommand{\inner}[2]{\left\langle#1, #2\right\rangle}
\newcommand{\norm}[1]{\left\|#1\right\|}
\newcommand{\opnorm}[1]{\left\|#1\right\|_{\mathrm{op}}}
\newcommand{\geqs}[1]{\gtrsim_{#1}}
\newcommand{\leqs}[1]{\lesssim_{#1}}

\newcommand{\A}{\ensuremath{\mathcal{A}}}
\newcommand{\B}{\ensuremath{\mathcal{B}}}

\newcommand{\blue}[1]{\textcolor{blue}{#1}}

\dajAUTHORdetails{%
  title = {New Bound for Roth's Theorem with Generalized Coefficients}, 
  author = {C\'edric Pilatte},
  plaintextauthor = {Cedric Pilatte},
  keywords = {Roth's theorem, discrete Fourier analysis, Bohr sets, density increments},
}   

\dajEDITORdetails{%
   year={2022},
   number={16},
   received={9 November 2021},   
   published={2 December 2022},  
   doi={10.19086/da.55553},       
}   

\begin{document}

\begin{frontmatter}[classification=text]

\title{New Bound for Roth's Theorem\\with Generalized Coefficients} 

\author[pil]{C\'edric Pilatte}

\begin{abstract}
We prove the following conjecture of Shkredov and Solymosi: every subset $A \subset \Z^2$ such that $\sum_{a\in A\setminus\{0\}} 1/\norm{a}^{2}   = +\infty$ contains the three vertices of an isosceles right triangle. To do this, we adapt the proof of the recent breakthrough by Bloom and Sisask on sets without three-term arithmetic progressions, to handle more general equations of the form $T_1a_1+T_2a_2+T_3a_3 = 0$ in a finite abelian group $G$, where the $T_i$'s are automorphisms of $G$. 
\end{abstract}
\end{frontmatter}

\section{Introduction}
\label{sec:intro}

In their 2020 breakthrough paper, Bloom and Sisask \cite{BS} improved the best known upper bound on the largest possible size of a subset of $\{1, 2, \dots, n\}$ without three-term arithmetic progression. They showed that, if $A \subset \{1, 2, \dots, n\}$ has no non-trivial three-term arithmetic progression, then 
\begin{equation*}
|A| \ll \frac{n}{(\log n)^{1+c}}
\end{equation*}
for some absolute constant $c>0$. The best previously available bound was ${n}/{(\log n)^{1-o(1)}}$, which had been obtained in four different ways \cite{sanders,bloomlondon,bsdiscrete,schoenrecent}.

Their result received a lot of attention as it settled the first interesting case of one of Erd\H{o}s' most famous conjectures. Erd\H{o}s conjectured that, if $A\subset \N$ is such that $\sum_{n\in A} 1/n$ diverges, then $A$ contains infinitely many $k$-term arithmetic progressions, for every $k\geq 3$. The result of Bloom and Sisask implies the case $k = 3$. The general case seems to be well beyond the reach of the current techniques.

The theorem of Bloom and Sisask can be applied to the prime numbers to recover a result of Green in analytic number theory. It is an old result of Van der Corput that the set of primes contains infinitely many three-term arithmetic progressions. Much more recently, Green \cite{green} generalized this fact to relatively dense subsets of the primes. The theorem of Bloom and Sisask gives a different proof of this, where Chebyshev's estimate $\pi(x) \gg x/\log{x}$ is the only fact about the primes that is used.

A three-term arithmetic progression is a solution to the equation $a_1 -2a_2 + a_3 = 0$. In this paper, we generalize the proof of Bloom and Sisask to deal with equations of the form $T_1a_1 +T_2a_2 + T_3a_3 = 0$ for an extended class of coefficients $T_1$, $T_2$ and $T_3$. More precisely, we prove the following in \cref{sec:pfmain}.

\begin{theorem}
\label{thm:main}
Let $G$ be a finite abelian group and let $T_1, T_2, T_3$ be automorphisms of $G$ such that $T_1 + T_2 + T_3 = 0$. If $A$ is a subset of $G$ without non-trivial solutions\footnote{A solution $(a_1, a_2, a_3)\in A^3$ is trivial if $a_1 = a_2 = a_3$.} to the equation 
\begin{equation}
\label{eq:ourequation}
T_1 a_1 + T_2 a_2 + T_3 a_3 = 0,
\end{equation}
then
\begin{equation*}
|A| \ll \frac{|G|}{(\log|G|)^{1+c}}
\end{equation*}
where $c > 0$ is an absolute constant.\footnote{In particular, the constant $c$ does not depend on $G$ or on the coefficients $T_i$.}
\end{theorem}

The result \cite[Corollary~3.2]{BS} of Bloom and Sisask corresponds to the special case ${T_1 = T_2 = \mathrm{Id}_{G}}$ and $T_3 = -2\, \mathrm{Id}_{G}$ of \cref{thm:main}. Their hypothesis that $G$ has odd order ensures that $-2\, \mathrm{Id}_{G}$ is an automorphism. 

\begin{remark}
\label{rem:translation}
The condition $T_1 + T_2 + T_3 = 0$ ensures that \cref{eq:ourequation} is translation-invariant. It is necessary: for example, if $G = \F_2^n$, ${T_1 = T_2 = T_3 = \mathrm{Id}_{G}}$ and $A$ is the set of vectors with first coordinate equal to $1$, then $|A| \asymp |G|$, yet $A$ has no solutions to \cref{eq:ourequation}.
\end{remark}

We will deduce the following corollary, which generalizes \cite[Corollary~1.2]{BS} to higher dimensions and matrix coefficients. It is also a strengthening of \cite[Theorem~2.21]{Bloomthesis}.

\begin{corollary}
\label{cor:matrices}
Let $M_1, M_2, M_3$ be nonsingular $d\times d$ matrices with integer coefficients such that $M_1 + M_2 + M_3 = 0$. If $A \subset \Z^d$ satisfies
\begin{equation*}
\sum_{a\in A\setminus\{0\}} \frac{1}{\norm{a}^d} = +\infty,
\end{equation*}
then $A$ contains infinitely many non-trivial solutions to the equation $M_1 a_1 + M_2 a_2 + M_3 a_3 = 0$.
\end{corollary}


Using \cref{cor:matrices}, we are able to prove a conjecture of Shkredov and Solymosi \cite[Conjecture~2]{SS}.

\begin{example}
\label{ex:IRtriangle}
If a subset $A$ of the square lattice satisfies $\sum_{a\in A\setminus\{0\}} {1}/{\norm{a}^2} = +\infty$, then there are infinitely many isosceles right triangles whose vertices are in $A$.
\end{example}

We also obtain the following aesthetic result.

\begin{example}
\label{ex:EQtriangle}
If a subset $A$ of the hexagonal lattice satisfies $\sum_{a\in A\setminus\{0\}} {1}/{\norm{a}^2} = +\infty$, then  $A$ contains infinitely many equilateral triangles.
\end{example}

\Cref{ex:EQtriangle,ex:IRtriangle} are special cases of the following corollary.


\begin{corollary}
\label{cor:generaltriangles}
Let $\Lambda \subset \C$ be a lattice of the form $\Lambda = \omega_1 \Z \oplus \omega_2 \Z$, such that $\omega_i \Lambda \subset \Lambda$ for $i=1,2$.
Let $T$ be any triangle with vertices in $\Lambda$. If $A \subset \Lambda$ is such that 
\begin{equation*}
\sum_{a\in A\setminus\{0\}} \frac{1}{|a|^2} = +\infty,
\end{equation*}
then there are infinitely many triangles, with vertices in $A$, which are directly similar\footnote{Two triangles are directly similar if there is an orientation-preserving similitude of the plane mapping one to the other.} to $T$.
\end{corollary}

\begin{proof}
The orientation-preserving similitudes of the plane are exactly the transformations of the form $z \mapsto uz+v$ with $u, v\in \C$, $u\neq 0$. Let $p_1, p_2$ and $p_3$ be the (distinct) vertices of~$T$. 

Finding triangles in $A$ that are directly similar to $T$ is equivalent to solving the system of equations 
\begin{equation*}
\left\{
\begin{array}{c}
up_1+v = a_1\\
up_2+v = a_2\\
up_3+v = a_3
\end{array} \right.
\end{equation*}
for $u\in \C\setminus \{0\}$, $v\in \C$ and $a_1, a_2, a_3\in A$. This system is equivalent to the single equation
\begin{equation*}
(p_3-p_2)a_1 + (p_1-p_3)a_2 + (p_2-p_1)a_3 = 0,
\end{equation*}
to be solved for distinct $a_1, a_2, a_3\in A$. 

Define $M_1$, $M_2$ and $M_3$ to be the matrices corresponding to multiplication by $p_3-p_2$, $p_1-p_3$ and $p_2-p_1$ in the $\Z$-basis $(\omega_1, \omega_2)$ of $\Lambda$. These matrices sum to zero, are nonsingular as the $p_i$'s are distinct, and have integer coefficients since $p_i\omega_j\in  \Lambda$ for all $i,j$. We conclude by \cref{cor:matrices}.
\end{proof}

%
%
%
%
%

\begin{remark}
It is believed that \cref{ex:IRtriangle} can be extended significantly: a conjecture of Graham states that, if $A\subset \Z^2$ is such that $\sum_{a\in A\setminus\{0\}} {1}/{\norm{a}^2} = +\infty$, then $A$ contains infinitely many axes-parallel squares \cite[Conjecture~8.4.6]{Gr}. The difficulty of Graham's conjecture is comparable to that of Erd\H{o}s' conjecture on arithmetic progressions of length $k=4$.
\end{remark}


\textbf{Overview of the paper.} In \cref{sec:matrix}, we will show how \cref{cor:matrices} follows from \cref{thm:main}. The rest of the paper will be devoted to the proof of \cref{thm:main}. 

Our proof is an adaptation of the work of Bloom and Sisask on three-term arithmetic progressions~\cite{BS}. We will use the same notation as in their paper. We will recall some of it in \cref{sec:notation}, where we also restate some classical lemmas that will be used throughout the proof. The more technical definitions, such as those of additively non-smoothing sets or of additive frameworks, can be found in \cite{BS}. 

The structure of the proof of \cref{thm:main} is shown in \cref{fig:dependency}. \Cref{sec:weaker} is dedicated to the proof of \cref{thm:weakerbound}, a result which by itself is sufficient to prove a weaker version of \cref{thm:main}, with the bound ${|G|}/{(\log |G|)^{1-o(1)}}$ instead. The proof of \cref{thm:weakerbound} is similar to that of \cref{thm:main}, but is considerably simpler. It uses a density increment lemma from \cite{Bloomthesis}. 

In \cref{sec:pfmain}, we prove \cref{thm:main} by adapting the work of Bloom and Sisask \cite{BS} to our more general setting. Fortunately, large portions of their paper can be used as a black box, without any modification. This is especially the case for \cite[Sections~9~and~10]{BS} (structure theorem for additively non-smoothing sets), as well as \cite[Section~11]{BS} (spectral boosting). We will mostly need to adapt some results from \cite[Sections 5,~8~and~12]{BS}.

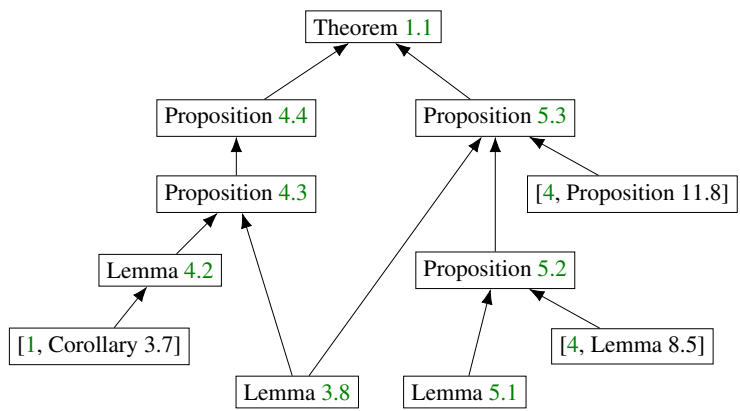
\begin{figure}[H]
\centering
\scalebox{0.81}{
\begin{tikzpicture}[scale=1]
		\node [draw] (0) at (-2.25, 3.5) {\cref{thm:weakerbound}}; 
		\node [draw] (1) at (-2.25, 2.25) {\cref{prop:weakerincrement}}; 
		\node [draw] (3) at (2, 3.5) {\cref{prop:incrementOrSolutions}}; 
		\node [draw] (4) at (-1.25, -1) {\cref{lem:bourgain}};
		\node [draw] (5) at (-3.5, 1) {\cref{cor:thesis}}; 
		\node [draw] (6) at (2, 1) {\cref{prop:incrementOrANS}}; 
		\node [draw] (7) at (4.25, -0.25) {\cite[Lemma 8.5]{BS}};
		\node [draw] (8) at (4.25, 2.25) {\cite[Proposition 11.8]{BS}};
		\node [draw] (9) at (1.5, -1) {\cref{lem:progressions}};  
		\node [draw] (10) at (0, 5) {\cref{thm:main}}; 
		
		\node [draw] (11) at (-4.5, -0.25) {\cite[Corollary~3.7]{Bloomthesis}};
		\draw [-{Latex[length=2.5mm]}] (11) to (5);
		
		\draw [-{Latex[length=2.5mm]}] (1) to (0);
		\draw [-{Latex[length=2.5mm]}] (5) to (1);
		\draw [-{Latex[length=2.5mm]}] (4) to (1);
		\draw [-{Latex[length=2.5mm]}] (4) to (3);
		\draw [-{Latex[length=2.5mm]}] (8) to (3);
		\draw [-{Latex[length=2.5mm]}] (6) to (3);
		\draw [-{Latex[length=2.5mm]}] (7) to (6);
		\draw [-{Latex[length=2.5mm]}] (9) to (6);
		\draw [-{Latex[length=2.5mm]}] (0) to (10);
		\draw [-{Latex[length=2.5mm]}] (3) to (10);
\end{tikzpicture}}
\caption{Dependency graph for the proof of \cref{thm:main} (only the main lemmas and propositions are shown).}
\label{fig:dependency}
\end{figure}

\smallbreak
\textbf{Comparison with the Bloom-Sisask proof.}
We strongly recommend the readers to familiarize themselves with the article of Bloom and Sisask before reading \cref{sec:notation,sec:weaker,sec:pfmain} of this paper. We have attempted to make as few changes to their proof as possible, to make the comparison easier for the reader. 

The proof of Bloom and Sisask can be immediately generalised to equations as in \cref{thm:main} for automorphisms that are multiples of the identity. If $T = n\, \mathrm{Id}_G$ and $B$ is a Bohr set, then the dilate $B_{\rho}$ is a subset of both $B$ and $T^{-1}B$, provided that $\rho \leq 1/n$ (see \cref{sec:notation} for the relevant definitions). This very useful property no longer holds for general automorphisms. 

Instead of considering a simple dilate $B_{\rho}$, we will need to work with the intersection $B \cap T^{-1}B$. The dilate of a Bohr set is another Bohr set of the same rank. By contrast, $B \cap T^{-1}B$ is still a Bohr set, but the rank may have doubled! Controlling the rank of these repeated intersections is the main additional difficulty. To overcome it, we need to keep track more explicitly of the frequency sets of all the Bohr sets in the proof.

Carefully tracking the dependence on the coefficients allows us to show that the rank of the successive Bohr sets in the density increment iteration grows polynomially. To obtain this, we also need to assume that the automorphisms $T_i$ commute (see \cref{rem:frequencysets} for more details). Since three-term equations always reduce to the case of commuting automorphisms (see \cref{rem:identity}), there is no commutativity assumption in \cref{thm:main}.

\Cref{thm:main} gives a bound to subsets of $\Z/N\Z$ without solutions to $ax+by+cz=0$ for integers $a,b,c$ comprime to $N$ with $a+b+c = 0$. It is important to note that this bound is \emph{uniform} in $a,b,c$. Such uniformity would not have been obtained through a `naive' modification of the Bloom-Sisask proof using dilates as above.

\begin{remark}
\label{rem:moreterms}
\Cref{thm:main} can be generalized to equations with more than three terms. More precisely, a slight adaptation of the proof shows the following. If $T_1, T_2, \dots, T_k$ are commuting automorphisms of an abelian group $G$ such that $T_1 + T_2 + \dots + T_k = 0$, then any set $A \subset G$ without non-trivial solutions to $T_1a_1 + T_2a_2 + \dots T_ka_k = 0$ satisfies 
\begin{equation*}
|A| \ll \frac{|G|}{(\log|G|)^{1+c}}
\end{equation*}
where $c > 0$ is an absolute constant. \Cref{cor:matrices} can also be modified in a similar way. However, for $k \geq 4$, considerably better bounds are available using other methods (see~\cite{schoen}), which is why we restrict ourselves to the case $k=3$.
\end{remark}

\section{Application to Matrix Coefficients}
\label{sec:matrix}

In this section, we show how \cref{cor:matrices} follows from \cref{thm:main}. The proof is standard and involves two steps: first truncating the set $A$, then embedding this truncation of $A$ inside a finite abelian group.

\begin{proof}[Proof of \cref{cor:matrices}]
Let $A$ be a subset of $\Z^d$ containing only finitely many non-trivial solutions to the equation 
\begin{equation}
\label{eq:progressions}
M_1 a_1 + M_2 a_2  + M_3 a_3 = 0.
\end{equation}
We want to prove that
\begin{equation}
\label{eq:finitesum}
\sum_{a\in A\setminus\{0\}} \frac{1}{\norm{a}^d} < +\infty.
\end{equation}
After removing a finite number of elements from $A$, we can assume that $A$ has no non-trivial solution to \cref{eq:progressions}.

For $T \geq 1$, let $A_T$  be the truncated set 
\begin{equation*}
A_T := A \cap [-T, T]^d.
\end{equation*}
It is sufficient to prove that, for all $T\geq 3$,
\begin{equation}
\label{eq:truncatedbound}
|A_T| \ll \frac{T^d}{(\log{T})^{1+c}},
\end{equation}
where $c > 0$ is the constant from \cref{thm:main}.\footnote{In this proof, the implied constants in the asymptotic notation $\ll$ and $\asymp$ depend only on the dimension $d$ and the matrices $M_1, M_2, M_3$.}
Indeed, we have
\begin{equation*}
\sum_{\substack{a\in A\setminus \{0\}\\ \norm{a}_{\infty} \leq M}} \frac{1}{\norm{a}^d} \asymp \sum_{N=1}^{M} \frac{1}{N^d}\cdot \# \{a\in A : \norm{a}_{\infty} = N\},
\end{equation*}
and by partial summation, together with \cref{eq:truncatedbound}, we get
\begin{equation*}
\sum_{\substack{a\in A\setminus \{0\}\\ \norm{a}_{\infty} \leq M}} \frac{1}{\norm{a}^d} \asymp \frac{|A_M|}{M^d} +  \int_{1}^{M} \frac{|A_T|}{T^{d+1}} \,\mathrm{d}T \ll 1 + \int_2^M \frac{1}{T (\log{T})^{1+c}} \,\mathrm{d}T \ll 1.
\end{equation*}
Taking $M \to +\infty$ proves \cref{eq:finitesum}.

Let $T\geq 3$. Let 
\begin{equation*}
C = \max\big(\opnorm{M_1}, \opnorm{M_2}, \opnorm{M_3}, \lvert \det M_1\rvert, \lvert\det M_2\rvert, \lvert\det M_3\rvert\big),
\end{equation*}
where $\opnorm{M_i}$ is the operator norm of the matrix $M_i$, viewed as a map $(\R^d, \norm{\cdot}_{\infty}) \to (\R^d, \norm{\cdot}_{\infty})$. Let $p$ be a prime number between $4CT$ and $8 CT$, which exists by Bertrand's postulate.

We embed $A_T$ in the abelian group $(\Z/p\Z)^d$. Let $\overline{A_T}$, $\overline{M_1}, \overline{M_2}$ and $\overline{M_3}$ be the reductions of $A_T, M_1, M_2$ and $M_3$ modulo $p$. Clearly, each $\overline{M_i}$ is invertible as its determinant is not divisible by~$p$.

We claim that the map
\begin{equation*}
\{(a_1, a_2, a_3) \in (A_T)^3 : M_1a_1 + M_2a_2 + M_3a_3 = 0\} \to \{(x_1, x_2, x_3) \in (\overline{A_T})^3 : \overline{M_1}x_1 + \overline{M_2}x_2 + \overline{M_3}x_3 = 0\} 
\end{equation*}
given by reduction modulo $p$ is surjective. Indeed, if $(a_1, a_2, a_3)\in (A_T)^3$ is such that 
\begin{equation*}
M_1a_1 + M_2a_2 + M_3a_3 \equiv 0 \pmod{p},
\end{equation*}
then $M_1a_1 + M_2a_2 + M_3a_3 = 0$ in $\R^d$ since we also have
\begin{equation*}
\norm{M_1a_1 + M_2a_2 + M_3a_3}_{\infty} \leq 3CT < p.
\end{equation*}
It follows that $\overline{A_T}$ only has trivial solutions to the equation $\overline{M_1}x_1 + \overline{M_2}x_2 + \overline{M_3}x_3 = 0$. By \cref{thm:main}, we obtain
\begin{equation*}
|A_T| = |\overline{A_T}| \ll \frac{p^d}{(\log p^d)^{1+c}} \asymp \frac{T^d}{(\log T)^{1+c}},
\end{equation*}
which proves \cref{eq:truncatedbound} and concludes the proof of \cref{cor:matrices}.
\end{proof}

\section{Notation and New Density Increments}
\label{sec:notation}

We use the same notation as in the paper of Bloom and Sisask \cite{BS}. We recall some of it below, but we encourage the readers to familiarize themselves with their article before reading the rest of this paper.

\begin{notation}
Fix a finite abelian group $G$. If $A\subset B$, the \emph{relative density} of $A$ in $B$ is the ratio $|A|/|B|$. If $X\subset G$, the density of $X$ in $G$ is denoted by $\mu(X) := |X|/|G|$. We write $\m{X}$ for the normalized indicator function $\m{X} = \mu(X)^{-1}\ind{X}$. 

For $f, g:G\to \C$, we use the normalizations 
\begin{equation*}
\inner{f}{g} := \frac{1}{|G|} \sum_{x\in G} f(x) \overline{g(x)} \quad \text{and}\quad f * g (x) := \frac{1}{|G|} \sum_{y\in G} f(y) {g(x-y)},
\end{equation*}
while for $f, g:\widehat{G}\to \C$, we set
\begin{equation*}
\inner{f}{g} := \sum_{\gamma\in \widehat{G}} f(\gamma) \overline{g(\gamma)}\quad\text{and}\quad f * g (x) :=\sum_{y\in \widehat{G}} f(y) {g(x-y)}.
\end{equation*}

In order to suppress logarithmic factors, we use the notation $X \leqs{\alpha} Y$ or $X = \tilde{O}_{\alpha}(Y)$ to mean that $|X| \leq C_1 \log(2/\alpha)^{C_2}Y$ for some constants $C_1, C_2 > 0$. 
\end{notation}

\begin{notation}[Bohr sets]
\label{not:bohr}
For $\Gamma\subset \widehat{G}$ and $\nu : \Gamma \to [0, 2]$, we define the \emph{Bohr set} $B = \mathrm{Bohr}_{\nu}(\Gamma)$ to be the subset of $G$ defined by
\begin{equation*}
\mathrm{Bohr}_{\nu}(\Gamma) = \{x\in G: |1-\gamma(x)| \leq \nu(\gamma) \text{ for all }\gamma \in \Gamma\}.
\end{equation*}
The set $\Gamma$ is called the \emph{frequency set} of $B$ and $\nu$ its \emph{width function}. The \emph{rank} of $B$, denoted by $\mathrm{rk}(B)$, is defined to be the size of $\Gamma$. Note that all Bohr sets are symmetric. 

When we speak of a Bohr set, we implicitly refer to the triple $(\mathrm{Bohr}_\nu(\Gamma), \Gamma, \nu)$, since the Bohr set $\mathrm{Bohr}_\nu(\Gamma)$ alone does not uniquely determine the frequency set nor the width.

The intersection of two Bohr sets is again a Bohr set. If $B = \mathrm{Bohr}_\nu(\Gamma)$ and $\rho > 0$, we denote by $B_{\rho}$ the \emph{dilate} of $B$, i.e.~the Bohr set given by $B_{\rho} := \mathrm{Bohr}_{\rho\nu}(\Gamma)$.

A Bohr set $B$ of rank $d$ is \emph{regular} if for all $|\kappa| \leq 1/(100d)$, we have 
\begin{equation*}
(1-100d|\kappa|) |B| \leq |B_{1+\kappa}| \leq (1+100d|\kappa|) |B|.
\end{equation*}
An important property is that, for every Bohr set $B$, there is a dilate $B_{\rho}$, for some $\rho \in [1/2, 1]$, which is regular (see \cite[Lemma~4.3]{BS}). 

If $B$ is a Bohr set and $T$ is an automorphism, then $TB$ is a Bohr set and $(TB)_{\rho} = TB_{\rho}$. If $B$ is regular, then so too is $TB$.
\end{notation}

The sizes of Bohr sets can be controlled using the classical lemma \cite[Lemma~4.4]{BS}. We restate it below as it will be used extensively throughout the article.

\begin{lemma}
\label{lem:boundbohr}
Let $\Gamma \subset \widehat{G}$ and $\nu, \nu' : \Gamma \to [0, 2]$ be such that $\nu'(\gamma) \leq \nu(\gamma)$ for $\gamma\in \Gamma$. We have 
\begin{equation*}
|\mathrm{Bohr}_{\nu'}(\Gamma)| \geq \Bigg(\prod_{\gamma\in \Gamma} \frac{\nu'(\gamma)}{4\nu(\gamma)}\Bigg) |\mathrm{Bohr}_{\nu}(\Gamma)|.
\end{equation*}
In particular, if $\rho \in (0, 1)$ and $B$ is a Bohr set of rank $d$, then $|B_{\rho}| \geq (\rho/4)^d |B|$.
\end{lemma}

\begin{remark}
\label{rem:frequencysets}
One of the main difficulties that arise when working with general automorphisms $T_i$ is that we often have to control intersections of Bohr sets such as $B' = T_1B \cap T_2B \cap T_3B$. If the Bohr set $B$ has frequency set $\Gamma = \{\gamma_i \mid i\in I\}$, then $B'$ can be viewed as a Bohr set with frequency set $\Gamma' = \{\gamma_i \circ T_j^{-1}\mid i\in I, j\in \{1,2,3\}\}$. If we don't know anything about the frequency set of $B$, then the best we can say about the rank of $B'$ is that 
\begin{equation*}
\mathrm{rk}(B') \leq 3\,\mathrm{rk}(B).
\end{equation*}
Suppose that $B_0$ is a Bohr set of rank $d$ and define, for $n\geq 0$, $B_{n+1} = T_1B_n \cap T_2B_n \cap T_3B_n$. Using the above bound would give an exponential growth for the ranks of these Bohr sets. Such a naive bound would be completely insufficient to prove \cref{thm:main}. However, we can note that 
\begin{equation*}
B_n = \bigcap_{T\in W_n} TB_0,
\end{equation*}
where $W_n$ is the set of all compositions of $n$ automorphisms from the set $\{T_1, T_2, T_3\}$. If we know $T_1$, $T_2$ and $T_3$ commute, then $|W_n|$ has polynomial growth and we can obtain an acceptable bound for the rank of $B_n$, namely 
\begin{equation*}
\mathrm{rk}(B_n) \ll n^2\,\mathrm{rk}(B_0).
\end{equation*}
In the density increment argument, we will need to be more explicit with the definitions of the Bohr sets, in order to carefully keep track of their ranks and frequency sets.
\end{remark}

In the light of \cref{rem:frequencysets}, to make the proof of Bloom and Sisask work for general coefficients, we need to change the definition of density increments (\cite[Definition~5.1]{BS}). 

\begin{definition}[Increments]
\label{def:increments}
Let $B$ be a regular Bohr set, and let $B' \subset B$ be a regular Bohr set of rank $d$. We say that $A\subset B$ of relative density $\alpha$ has a \emph{density increment} of strength $[\delta, d'; C]$ relative to $B'$ if there is a regular Bohr set $B''$ of the form 
\begin{equation*}
B'' = B_{\rho}' \cap \widetilde{B}
\end{equation*}
such that 
\begin{equation*}
\norm{\ind{A} * \m{B''}}_{\infty} \geq (1+C^{-1}\delta) \alpha,
\end{equation*}
where $\widetilde{B} = \mathrm{Bohr}(\widetilde{\Gamma}, \widetilde{\nu})$ is a Bohr set of rank $|\widetilde{\Gamma}| \leq C d'$, $\rho \in (0, 1]$, and $\rho, \widetilde{\nu}$ satisfy the inequality
\begin{equation}
\label{eq:defsizebound}
\left(\frac{\rho}{4}\right)^d \prod_{\gamma \in \widetilde{\Gamma}} \frac{\widetilde{\nu}(\gamma)}{8} \geq (2d(d'+1))^{-C(d+d')}.
\end{equation}
\end{definition}

\begin{remark}
\label{rem:equivdensity}
If $A\subset B$ has a density increment of strength $[\delta, d'; C]$ with respect to $B'$ in the sense of \cref{def:increments}, then $A$ has a density increment of the same strength with respect to $B'$ in the sense of \cite[Definition~5.1]{BS}. This is because \cref{eq:defsizebound} implies the bound
\begin{equation}
\label{eq:formerdefsizebound}
|B''| \geq (2d(d'+1))^{-C(d+d')} |B'|,
\end{equation}
by a direct application of \cref{lem:boundbohr}.

The converse is not true in general, but it is true for all the density increments present in \cite{BS}. That is, every density increment in \cite{BS} is also a density increment in the sense of \cref{def:increments}, of the same strength. The reason is that
\begin{enumerate}
	\item the Bohr set $B''$ in \cite[Definition~5.1]{BS} is always chosen to be of the form $B'' = B_{\rho}' \cap \widetilde{B}$ in \cite{BS},
	\item and every time the authors show that some set $A\subset B$ has a density increment, they need to prove \cref{eq:formerdefsizebound}. To do this, the only tool they use is \cref{lem:boundbohr}, and thus they prove the stronger \cref{eq:defsizebound}.
\end{enumerate}
\end{remark}

We restate here \cite[Lemma~5.2]{BS}, which is an easy consequence of the definition of density increment.

\begin{lemma}
\label{lem:relativeincrements}
Let $B$ be a regular Bohr set and $B'\subset B$ be a regular Bohr set of rank $d$. Let $\rho \in (0, 1)$. If $A\subset B$ has a density increment of strength $[\delta, d'; C]$ relative to $B'_{\rho/d}$, then $A$ has a density increment of strength $[\delta, d'; C+\tilde{O}_{\rho}(1)]$ relative to $B'$.
\end{lemma}

Finally, we reproduce the statement of \cite[Lemma~12.1]{BS} for three smaller Bohr sets instead of two. The proof of \cref{lem:bourgain} is the same as that of \cite[Lemma~12.1]{BS}, so we shall not repeat it here.

\begin{lemma}
\label{lem:bourgain}
There is a constant $c>0$ such that the following holds. Let $\B$ be a regular Bohr set of rank $d$, let $\A\subset \B$ have relative density $\alpha$, let $\eps>0$ and suppose that ${B_1, B_2, B_3} \subset \B_\rho$ where $\rho \leq c\alpha \eps/d$. Then either 
\begin{enumerate}
	\item ($\A$ has almost full density on $B_1$, $B_2$ and $B_3$) there is an $x\in \B$ such that 
	\begin{equation*}
	\ind{\A} * \m{B_i} (x) \geq (1-\eps)\alpha
	\end{equation*}
	for $i=1, 2, 3$, or
	\item (density increment) $\A$ has an increment of strength $[\eps, 0; O(1)]$ relative to one of the $B_i$'s.
\end{enumerate}
\end{lemma}

\section{Proof of a Weaker Bound}
\label{sec:weaker}

In this section, we prove \cref{thm:weakerbound}, which can be regarded as a weaker version of \cref{thm:main}. On its own, \cref{thm:weakerbound} is sufficient to prove the bound 
\begin{equation*}
|A| \ll \frac{|G|}{(\log |G|)^{1-o(1)}},
\end{equation*}
keeping the notation of \cref{thm:main}. We will use \Cref{thm:weakerbound} at the end of the proof of \cref{thm:main} when, after a series of density increments, we arrive at a subset $A'$ of a Bohr set $B'$ whose relative density is substantially larger than the original density $|A|/|G|$.

\begin{remark}
\label{rem:identity}
It suffices to prove \cref{thm:main} when the first automorphism $T_1$ is the identity, something which we will assume from this point onward. To deduce the case of a general automorphism $T_1$, simply apply \cref{thm:main} to the set $T_1A$ and the automorphisms $\mathrm{Id_G}, T_2T_1^{-1}, T_3T_1^{-1}$. 
\end{remark}

From now on, fix two automorphisms $T_2$, $T_3$ of $G$ such that $\mathrm{Id}_G + T_2 + T_3 = 0$. We count the number of solutions to the equation $a_1+T_2a_2 + T_3a_3 = 0$ via the inner product
\begin{equation*}
T(A_1, A_2, A_3) := \inner{\ind{A_1} * \ind{T_2A_2}}{\ind{-T_3A_3}},
\end{equation*}
defined for $A_1, A_2, A_3 \subset G$. Observe that
\begin{equation*}
T(A, A, A) = \frac{1}{|G|^2}\cdot \# \{(a_1, a_2, a_3)\in A^3 : a_1+T_2a_2 + T_3a_3 = 0\}.
\end{equation*}

We will obtain \cref{thm:weakerbound} by repeated applications of the following lemma, which is a restatement of \cite[Corollary~3.7]{Bloomthesis} in the language of regular Bohr sets.

\begin{lemma}
\label{cor:thesis}
There is a constant $c\in (0, \tfrac12)$ such that the following holds. Let $0< \alpha\leq 1$. Let~$B$ be a regular Bohr set of rank $d$ and $B'$ a regular Bohr set of rank $\leq 3d$ such that $B' \subset B_{\rho}$, where $\rho = c \alpha/d$. Suppose that $A_1 \subset B$, $A_2\subset T_2^{-1}B'$ and $A_3\subset T_3^{-1}B$, each time with relative density at least $\alpha$. Then 
\begin{enumerate}
	\item either 
\begin{equation*}
T(A_1, A_2, A_3) \gg \alpha^3 \mu(B) \mu(B')
\end{equation*}
	\item or there is a regular Bohr set $B''$ such that $\norm{1_A * \m{B''}}_{\infty} \geq (1+c)\alpha$, where 
	\begin{itemize}
		\item $A$ is either $A_1$ or $-T_3A_3$,
		\item and $B''$ is of the form 
		\begin{equation*}
			B'' = B'_{\eta} \cap \widetilde{B}
		\end{equation*}
		for some $\eta \asymp \exp(-\tilde{O}_{\alpha}(1))/d$ and some $\widetilde{B} = \mathrm{Bohr}(\widetilde{\Gamma}, \widetilde{\nu})$ with $|\widetilde{\Gamma}| \leq d'/4$ and $\widetilde{\nu} \geq 1/d'$ on $\widetilde{\Gamma}$, where $d' = \tilde{O}_{\alpha}(\alpha^{-1})$.
	\end{itemize}
\end{enumerate}
\end{lemma}

\begin{proof}
This follows directly from \cite[Corollary 3.7]{Bloomthesis}, applied to the sets $A_1$, $T_2A_2$ and $-T_3A_3$. Note that, since $B$ is a regular Bohr set,
\begin{equation*}
|(B+B')\setminus B| \leq |B_{1+\rho} \setminus B| \leq (1+O(d\rho)) |B|,
\end{equation*}
so that $B'$ is $(2^{-270}\alpha)$-sheltered by $B$, provided that $c$ is sufficiently small (see \cite{Bloomthesis} for the definition of `sheltered' in this context). We see in a similar way that $B''$ has the required amount of shelter. Finally, if $x\in B''$ and $\gamma' \in \langle \widetilde{\Gamma}\rangle$, say $\gamma' = \sum_{\gamma\in \Gamma} \eps_{\gamma}\gamma$ with $ \eps_{\gamma}\in \{-1, 0, 1\}$, we have 
\begin{equation*}
\lvert 1-\gamma'(x)\rvert \leq \sum_{\gamma\in \widetilde{\Gamma}} \lvert 1-\gamma(x)^{\eps_\gamma}\rvert \leq \tfrac14
\end{equation*} 
as required.
\end{proof}

\begin{proposition}
\label{prop:weakerincrement}
Let $\B$ be a regular Bohr set of rank $d$, and let $\A \subset \B$ of relative density $\alpha$. Let~$B^{\star} = {\B \cap T_2\B \cap T_3\B}$. Then, either 
\begin{equation*}
T(\A, \A, \A) \gg \exp \left(-\tilde{O}_{\alpha}(d \log(2d)\right) \mu(B^{\star})^2
\end{equation*}
or $\A$ has a density increment of strength
\begin{equation*}
[1, \alpha^{-1}; \tilde{O}_{\alpha}(1)]
\end{equation*}
with respect to either $B^{\star}$, $T_2^{-1}B^{\star}$ or $T_3^{-1}B^{\star}$.
\end{proposition}

\begin{proof}
Let $\eps = c/2$, where $c$ is the constant of \cref{cor:thesis}. We apply \cref{lem:bourgain} with $B_1 = (B^{\star})_{\rho}$, $B_2 = T_2^{-1}(B^{\star})_{\rho\rho'}$, $B_3 = T_3^{-1}(B^{\star})_{\rho}$, where $\rho = c_1\alpha \eps/d$ and $\rho' = c_2\alpha/d$, with $c_1, c_2>0$ being two small constants, chosen in particular such that $B_1$, $B_2$ and $B_3$ are regular.

If the second case of \cref{lem:bourgain} holds, then $\A$ has a density increment of strength $[1, 0; O(1)]$ relative to one of the $B_i$'s. By \cref{lem:relativeincrements}, this implies that $\A$ has a density increment of strength $[1, 0; \tilde{O}_{\alpha}(1)]$ relative to $B^{\star}$, $T_2^{-1}B^{\star}$ or $T_3^{-1}B^{\star}$.

We may thus suppose that the first case of \cref{lem:bourgain} holds. That is, there is some $x\in G$ such that, if we let 
\begin{equation*}
A_1 = (\A - x) \cap B_1, \quad A_2 = (\A - x) \cap B_2\quad \text{and}\quad A_3 = (\A - x) \cap B_3,
\end{equation*}
then each $A_i$ has relative density at least $(1-\eps) \alpha$ in the corresponding $B_i$. We now use \cref{cor:thesis} with $B = (B^{\star})_{\rho}$ and $B' = (B^{\star})_{\rho\rho'}$. 

\begin{enumerate}
	\item In the first case, we get
        \begin{equation*}
        T(A_1, A_2, A_3) \gg (1-\eps)^3\alpha^3 \mu(B)  \mu(B') \gg \exp \left(-\tilde{O}_{\alpha}(d \log(2d)\right) \mu(B^{\star})^2,
        \end{equation*}
        where the last inequality follows from \cref{lem:boundbohr}. Since the $A_i$'s are subsets of the same translate of $\A$ and the equation $a_1+T_2a_2+T_3a_3 = 0$ is translation-invariant, we have 
        \begin{equation*}
        {T(\A, \A, \A) \geq T(A_1, A_2, A_3)},
        \end{equation*}
        which gives the claimed bound. 

	\item In the second case, there is a regular Bohr set $B''$ as in the statement of the lemma such that 
        \begin{equation}
        \label{eq:epsilonchoice}
        \norm{1_{A} * \m{B''}}_{\infty} \geq (1+c)(1-\eps) \alpha \geq (1+c/4)\alpha,
        \end{equation}
        where $A$ is either $A_1$ or $-T_3A_3$. We therefore deduce that $\A$ has a density increment of strength $[1, \alpha^{-1}; \tilde{O}_{\alpha}(1)]$ relative to $(B^{\star})_{\rho\rho'}$ or $T_3^{-1}(B^{\star})_{\rho\rho'}$. By \cref{lem:relativeincrements}, this means that $\A$ has a density increment of the same strength relative to $B^{\star}$ or $T_3^{-1}B^{\star}$.\qedhere
\end{enumerate}
\end{proof}

We now iteratively apply \cref{prop:weakerincrement} to obtain \cref{thm:weakerbound}, which plays the same role as \cite[Theorem~5.4]{BS} in the proof of Bloom and Sisask.

\begin{proposition}
\label{thm:weakerbound}
Let $B = \mathrm{Bohr}(\Gamma, \nu)$ be a regular Bohr set of rank $d$ and suppose that $A\subset B$ has density $\alpha$. Then
\begin{equation*}
T(A, A, A) \geq \exp \left(-\tilde{O}_{\alpha}(d + \alpha^{-1}) \log{2d}\right)\Bigg( \prod_{\gamma\in \Gamma}\frac{\nu(\gamma)}{8}\Bigg)^{\tilde{O}_{\alpha}(1)}.
\end{equation*}
\end{proposition}

\begin{proof}
Let $C = \tilde{O}_{\alpha}(1)$ be the constant in the density increment case of \cref{prop:weakerincrement} ($C$ is fixed as $\alpha$ is given). Recall that $\tilde{O}_{\alpha'}(1)$ is short for $C_1 \log(2/\alpha')^{C_2}$, which is a decreasing function of $\alpha'$. Thus, if we use \cref{prop:weakerincrement} with some pair $\A' \subset \B'$ having relative density $\alpha' \geq \alpha$ and the second case applies, we will have a density increment of strength $[1, (\alpha')^{-1}; C]$.

We inductively construct two sequences $(A_n)$ and $(B_n)$, where, for each $n$, $A_n$ is a subset of $B_n$ with relative density $\alpha_n$. Let $A_0 = A$ and $B_0 = B$. Assume that $A_i$ and $B_i$ have been constructed for $i<n$. We use \cref{prop:weakerincrement} with $\A = A_{n-1}$ and $\B = B_{n-1}$. If the first case of the proposition holds, we stop the construction. Otherwise, we are in the density increment case and there are sets $A_n \subset B_n$ such that
\begin{itemize}
	\item $A_{n}$ is a subset of a translate of $A_{n-1}$;
	\item $A_n$ is a subset of $B_n$ of relative density $\alpha_n \geq (1+C^{-1})\alpha_{n-1}$;
	\item $B_n$ is a regular Bohr set of the form 
	\begin{equation}
		\label{eq:inters1weaker}
		B_n = (S_n B_{n-1}^{\star})_{\rho_n} \cap \widetilde{B}_{n},
	\end{equation}
	where 
	\begin{itemize}
		\item $B_{n-1}^{\star}$ is the Bohr set
		\begin{equation}
		\label{eq:inters2weaker}
		B_{n-1}^{\star}:= B_{n-1} \cap T_2 B_{n-1}\cap T_3 B_{n-1},
		\end{equation}
		whose rank we denote by $d_{n-1}^{\star}$,
		\item $S_n$ is either $\mathrm{Id}_G$, $T_2^{-1}$ or $T_3^{-1}$,
		\item $\widetilde{B}_n = \mathrm{Bohr}(\widetilde{\Gamma}_n, \widetilde{\nu}_n)$, where $|\widetilde{\Gamma}_{n}| \leq C\alpha^{-1}$ and 
	\begin{equation}
	\label{eq:boundwidthweaker}
	\bigg(\frac{\rho_n}{4}\bigg)^{d_{n-1}^{\star}}\prod_{\gamma \in \widetilde{\Gamma}_n} \frac{\widetilde{\nu}_n(\gamma)}{8} \geq  \exp \left(-\tilde{O}_{\alpha} \left(d_{n-1}^{\star}+\alpha^{-1}\right) \log\left(d_{n-1}^{\star}\right)\right).
	\end{equation}
	\end{itemize}
\end{itemize}

Note that, since $1\geq \alpha_n \geq (1+C^{-1})^n\alpha$, this construction must terminate in $l = \tilde{O}_{\alpha}(1)$ steps. We then arrive at $A_l\subset B_l$, for which 
\begin{equation}
\label{eq:TAAAweaker}
T(A, A, A) \geq T(A_l, A_l, A_l) \gg \exp \left(-\tilde{O}_{\alpha} \left(d_l \log(2d_l)\right) \right) \mu(B_l^{\star})^2,
\end{equation}
where $d_l$ is the rank of $B_l$ and, as usual, $B_l^{\star} = B_{l} \cap T_2 B_{l}\cap T_3 B_{l}$.

Let $W_i$ be the set of all automorphisms obtained by composing $i$ elements of $\{\mathrm{Id}_{G}, T_2, T_3, T_2^{-1}, T_3^{-1}\}$. Since $T_3 = -\mathrm{Id}_{G} - T_2$, these automorphisms commute, which implies that $|W_i| \leq (2i+1)^2$.

An immediate induction using \cref{eq:inters1weaker,eq:inters2weaker} shows that 
\begin{equation}
\label{eq:Bohrsetsweaker}
B_i^{\star} \supset \Bigg(\bigcap_{T\in W_{2i+1}} (TB)_{\rho_1\cdots \rho_i}\Bigg)  \cap\Bigg( \bigcap_{j=1}^{i} \bigcap_{T\in W_{2i+1}} (T \widetilde{B}_j)_{\rho_{j+1}\cdots \rho_i} \Bigg)
\end{equation}
for $0\leq i\leq l$. The same reasoning shows that the frequency set of $B_i$ is contained in 
\begin{equation*}
\Bigg(\bigcup_{T\in W_{2i}} T\Gamma\Bigg)  \cup\Bigg( \bigcup_{j=1}^{i} \bigcup_{T\in W_{2i}} T \widetilde{\Gamma}_j\Bigg).
\end{equation*}
This shows that
\begin{equation*}
d_{l} \ll l^2 d + l^3 C\alpha^{-1} = \tilde{O}_{\alpha}(d  + \alpha^{-1}).
\end{equation*}
In particular, \cref{eq:boundwidthweaker} becomes
\begin{equation}
\label{eq:boundwidthweaker3}
	\bigg(\frac{\rho_n}{4}\bigg)^{d_{n-1}^{\star}}\prod_{\gamma \in \widetilde{\Gamma}_n} \frac{\widetilde{\nu}_n(\gamma)}{8} \geq   \exp \left( -\tilde{O}_{\alpha} \left(d+\alpha^{-1}\right)\log (2d) \right),
\end{equation}
for $1\leq n \leq l$. 

We now use \cref{lem:boundbohr} to give a lower bound for $\mu(B_l^{\star})$. By \cref{eq:Bohrsetsweaker}, we have 
\begin{equation*}
\mu(B_l^{\star}) \geq \left(\Bigg(\frac{1}{4}\prod_{j=1}^{l} \rho_j\Bigg)^d\prod_{\gamma\in \Gamma} \frac{\nu(\gamma)}{8}\right)^{|W_{2l+1}|} \cdot \prod_{j=1}^{l}\left( \Bigg(\frac{1}{4}\prod_{n=j+1}^{l} \rho_n\Bigg)^{\mathrm{rk}(\widetilde{B}_j)}\prod_{\gamma\in \widetilde{\Gamma}_j} \frac{\widetilde{\nu}_j(\gamma)}{8}\right)^{|W_{2l+1}|}.
\end{equation*}
Using \cref{eq:boundwidthweaker3} and the simple inequalities $d\leq d_{j-1}^{\star}$ and $\mathrm{rk}(\widetilde{B}_j) \leq d_{n-1}^{\star}$ for $j\geq 1$ and $j+1\leq n\leq l$, this yields
\begin{equation*}
\mu(B_l^{\star}) \geq  \exp \left( -\tilde{O}_{\alpha} \left(d+\alpha^{-1}\right)\log (2d) \right) \Bigg(\prod_{\gamma\in \Gamma} \frac{\nu(\gamma)}{8}\Bigg)^{\tilde{O}_{\alpha}(1)}.
\end{equation*}
Together with \cref{eq:TAAAweaker}, this concludes the proof of \cref{thm:weakerbound}.
\end{proof}

\section{Proof of the Main Theorem}
\label{sec:pfmain}

This section is dedicated to the proof of \cref{thm:main}. Each statement in this section is an adaptation of a corresponding statement in \cite{BS}. To help the reader, we will highlight the changes made to the original statements of \cite{BS} in blue.

We start by proving an analogue of \cite[Lemma~8.2]{BS} in our setting. When $A \subset B$, the notation $\m{A/B}$ stands for the \emph{balanced function} $\m{A/B} := \m{A} - \m{B}$.

\begin{lemma}
\label{lem:progressions}
There is a constant $c > 0$ such that the following holds. Let $0 < \alpha \leq 1$. Let $B$ be a regular Bohr set of rank $d$, and $B'$ another regular Bohr set such that \blue{$T_3B' \subset B_{\rho}$}, with $\rho \leq c\alpha/d$. Let \blue{$A_1 \subset B$, $A_2\subset T_2^{-1}B$} and $A_3 \subset B'$, each time with relative density in $[\alpha/2, 2\alpha]$. Then either 
\begin{enumerate}
	\item (many solutions) $\blue{T(A_1, A_2, A_3)} \geq \tfrac{1}{16} \alpha^3 \mu(B) \mu(B')$, or
	\item (large $L^2$ mass on a spectrum) there is some $\eta \gg \alpha$ such that
	\begin{equation*}
	\sum_{\gamma \in \Delta_\eta (\blue{-T_3A_3})} |\mh{A/B}(\gamma)|^2 \gtrsim_{\alpha} \eta^{-1} \mu(B)^{-1},
	\end{equation*}
	where \blue{$A$ is either $A_1$ or $T_2A_2$}.
\end{enumerate}
\end{lemma}

\begin{proof}
We have 
\begin{equation*}
T(A_1, A_2, A_3) = \inner{\ind{A_1} * \ind{T_2A_2}}{\ind{-T_3A_3}} \geq \tfrac{1}{8}\alpha^3 \mu(B)^2\mu(B') \inner{\m{A_1} * \m{T_2A_2}}{\m{-T_3A_3}}.
\end{equation*}
Replacing $\m{A_1}$ and $\m{T_2A_2}$ with their balanced functions $\m{A_1/B}$ and $\m{T_2A_2/B}$, we have 
\begin{align*}
\inner{\m{A_1} * \m{T_2A_2}}{\m{-T_3A_3}} &= \inner{\m{A_1/B} * \m{T_2A_2/B}}{\m{-T_3A_3}} + E,
\end{align*}
where
\begin{equation*}
E = \inner{\m{A_1} * \m{B}}{\m{-T_3A_3}}+ \inner{\m{B} * \m{T_2A_2}}{\m{-T_3A_3}}- \inner{\m{B} * \m{B}}{\m{-T_3A_3}}.
\end{equation*}
We can estimate $E$ using regularity. Since $-T_3A_3\subset B_{\rho}$, we have $\norm{\m{-T_3A_3} * \m{B} - \m{B}}_1 = O(\rho d)$ by \cite[Lemma~4.5]{BS}. Moreover, $\norm{\m{A_1}}_{\infty}, \norm{\m{T_2A_2}}_{\infty}, \norm{\m{B}}_{\infty} \leq 2\alpha^{-1}\mu(B)^{-1}$. Therefore
\begin{align*}
E&= \inner{\m{A_1} }{\m{-T_3A_3} * \m{B}} + \inner{ \m{T_2A_2}}{\m{-T_3A_3}* \m{B}}- \inner{\m{B}}{\m{-T_3A_3}* \m{B}}\\
&= \inner{\m{A_1} }{\m{B}} + \inner{ \m{T_2A_2}}{\m{B}}- \inner{\m{B}}{\m{B}} + O(\rho d\alpha^{-1} \mu(B)^{-1})\\
&= \mu(B)^{-1} + \mu(B)^{-1}  - \mu(B)^{-1} +O\big(\rho d\alpha^{-1} \mu(B)^{-1}\big)\\
&= \mu(B)^{-1} + O\left(\rho d\alpha^{-1} \mu(B)^{-1}\right).
\end{align*}
In particular, $E\geq \frac{3}{4}\mu(B)^{-1}$, provided $\rho$ is small enough. Thus 
\begin{equation*}
T(A_1, A_2, A_3) \geq \tfrac{1}{8}\alpha^3 \mu(B)^2\mu(B') \Big(\inner{\m{A_1/B} * \m{T_2A_2/B}}{\m{-T_3A_3}} + \tfrac{3}{4}\mu(B)^{-1} \Big)
\end{equation*}
If the first case of the conclusion doesn't hold, then
\begin{equation*}
\inner{\m{A_1/B} * \m{T_2A_2/B}}{\m{-T_3A_3}} \leq -\tfrac{1}{4} \mu(B)^{-1}.
\end{equation*}
By Parseval's identity, followed by the triangle inequality, we deduce that
\begin{equation*}
\inner{|\mh{A_1/B}| |\mh{T_2A_2/B}|}{|\mh{-T_3A_3}|}\geq \tfrac{1}{4} \mu(B)^{-1}.
\end{equation*}
Using $xy\leq \tfrac12 (x^2+y^2)$, we find that 
\begin{equation*}
\sum_{\gamma \in \widehat{G}} |\mh{A/B}(\gamma)|^2|\mh{-T_3A_3}(\gamma)| = \inner{|\mh{A/B}|^2}{|\mh{-T_3A_3}|}\geq \tfrac{1}{4} \mu(B)^{-1},
\end{equation*}
where $A$ is either $A_1$ or $T_2A_2$. Since $\norm{\m{A/B}}_2^2\leq \alpha^{-1}\mu(B)^{-1}$, we can discard the terms of the above sum with $|\mh{-T_3A_3}| \leq \tfrac{1}{8}\alpha$ to obtain
\begin{equation*}
\sum_{\gamma \in \Delta_{\alpha/8}(-T_3A_3)} |\mh{A/B}(\gamma)|^2|\mh{-T_3A_3}(\gamma)| \geq \tfrac{1}{8} \mu(B)^{-1}.
\end{equation*}
By the dyadic pigeonhole principle, we conclude that there is some $1 \geq \eta \gg \alpha$ such that 
\begin{equation*}
\sum_{\gamma \in \Delta_{\eta}(-T_3A_3) \setminus  \Delta_{2\eta}(-T_3A_3) } |\mh{A/B}(\gamma)|^2|\mh{-T_3A_3}(\gamma)| \geqs{\alpha} \mu(B)^{-1}.
\end{equation*}
This concludes the proof since $|\mh{-T_3A_3}(\gamma)| \asymp \eta$ on the set $\Delta_{\eta}(-T_3A_3) \setminus  \Delta_{2\eta}(-T_3A_3)$.
\end{proof}

Next, we modify the statement of \cite[Proposition~8.1]{BS} as follows.

\begin{proposition}
\label{prop:incrementOrANS}
There is a constant $c>0$ such that the following holds. Let $k, h, t\geq 20$ be some parameters. 

Let $0< \alpha\leq 1$. Let $B$ be a regular Bohr set of rank $d$, and $B'$ another regular Bohr set, of \blue{rank at most $3d$}, such that \blue{$B' \subset T_3^{-1}B_{\rho}$}, where $\rho \leq c\alpha^2/d$. Let \blue{$A_1 \subset B$, $A_2 \subset T_2^{-1}B$} and $A_3 \subset B'$, each time with relative density in $[\alpha/2, 2\alpha]$. Then \blue{for either $A = A_1$ or $A = T_2A_2$}, one of the following holds 
\begin{enumerate}
	\item (large density) $\alpha \gg 1/k^2$, or
	\item (many solutions) $\blue{T(A_1, A_2, A_3)} \gg \alpha^3 \mu(B) \mu(B')$, or
	\item $A$ has a density increment of strength either 
	\begin{enumerate}
		\item (small increment) $[1, \alpha^{-1/k}; \tilde{O}_{\alpha}(h \log t)]$ or 
		\item (large increment) $[\alpha^{-1/k}, \alpha^{-1+1/k}; \tilde{O}_{\alpha}(h \log t)]$
	\end{enumerate}
	relative to $\blue{T_3B'}$, or
	\item (non-smoothing large spectrum) there is a set $\Delta$ and three quantities ${\rho_{\mathrm{top}}, \rho_{\mathrm{bottom}}, \rho'\in (0, 1)}$ satisfying
	\begin{equation*}
	\rho_{\mathrm{top}} \gg \alpha^{O(1)}(c/dt)^{O(h)}, \quad \rho_{\mathrm{bottom}}\gg (\alpha/d)^{O(1)}, \quad \text{and}\quad \rho' \gg (\alpha/d)^{O(1)},
	\end{equation*}
	such that 
	\begin{enumerate}
		\item $\alpha^{-3 + O(1/k)} \ll |\Delta| \leqs{\alpha} \alpha^{-3}$,
		\item there exists an additive framework $\tilde{\Gamma}$ of height $h$ and tolerance $t$ between 
		\begin{equation*}
		\Gamma_{\mathrm{top}} := \Delta_{1/2}(\blue{T_3 B_{\rho_{\mathrm{top}}}'})\quad \text{and}\quad\Gamma_{\mathrm{bottom}} := \Delta_{1/2}(\blue{T_3 B_{\rho_{\mathrm{bottom}}}'}),
		\end{equation*}
		\item $\Delta$ is $\tfrac{1}{4}$-robustly $(\tau, k')$-additively non-smoothing relative to $\tilde{\Gamma}$ for some $\alpha^{2-O(1/k)} \gg \tau \gg \alpha^{2+O(1/k)}$ and $k\geq k' \gg k$, and 
		\item if we let $B'' = (\blue{T_3 B_{\rho_{\mathrm{top}}}'})_{\rho'}$ then for all $\gamma \in\Delta + \Gamma_{\mathrm{top}}$
		\begin{equation*}
		|\mh{A/B}|^2 \circ |\mh{B''}|^2 (\gamma) \gg \alpha^{2+O(1/k)} \mu(B)^{-1},
		\end{equation*}
		and
		\item \begin{equation*}
		\Big\|\ind{\Delta} \circ |\mh{B''}|^2\Big\|_{\infty} \leq 2.
		\end{equation*}
	\end{enumerate}
\end{enumerate}
\end{proposition}

Few changes have to be made to the proof of \cite[Proposition~8.1]{BS}, so we only give an overview of the modified proof.

\begin{proof}[Proof sketch]
We keep the notation $B^{(0)} := B'$ and $B^{(i+1)} = B^{(i)}_{\rho_i} $ for some $\rho_i$ that are the same as those in the original proof.

By \cref{lem:progressions}, either we are in the second case or there is some $\eta \gg \alpha$ such that
\begin{equation*}
	\sum_{\gamma \in \Delta_\eta (-T_3A_3)} |\mh{A/B}(\gamma)|^2 \gtrsim_{\alpha} \eta^{-1} \mu(B)^{-1},
\end{equation*}
where $A$ is either $A_1$ or $T_2A_2$.

Suppose first that this is true for some $\eta \geq \tfrac12 K^{-1}$. In this case we apply \cite[Corollary~7.11]{BS} with $T_3B'$ instead of $B$, with $T_3B^{(1)}$ in place of $B'$, the function $f$ chosen to be $\ind{-T_3A_3}$ and  and the weight function $\omega$ given by $\omega = |\mh{A/B}|^2$, restricted to $\Delta_{\eta}(-T_3A_3)$. We apply \cite[Lemma~7.8]{BS} and \cite[Lemma~5.7]{BS} in the same way as in the original proof, except that we obtain a small density increment for $A$ relative to $T_3B'$ instead of $B'$. 

The case $ \tfrac12 K^{-1}\geq \eta \geq K^2\alpha$ is similar. After using \cite[Corollary~7.12]{BS}, \cite[Lemma~7.8]{BS} and \cite[Lemma~5.7]{BS}, we conclude that $A$ has a large increment relative to $T_3B'$.

Finally, in the case $\alpha \ll \eta \ll K^2\alpha$, we have 
\begin{equation*}
\sum_{\gamma \in \tilde{\Delta}} |\mh{A/B}(\gamma)|^2 \geqs{\alpha} K^2\alpha^{-1} \mu(B)^{-1},
\end{equation*}
where $\tilde{\Delta} = \Delta_{c\alpha}(-T_3A_3)$ for some absolute constant $c>0$. We use \cite[Lemma~6.2]{BS} to construct an additive framework between $\Gamma_{\mathrm{top}} = \Delta_{1/2}(T_3B^{(2)})$ and $\Gamma_{\mathrm{bottom}} = \Delta_{1/2}(T_3B^{(1)})$. Next, we use \cite[Lemma~8.5]{BS} with $A'$ being replaced by $-T_3A_3$, $B'$ being replaced by $T_3B'$, $B^{(1)}$ being replaced by $T_3B^{(2)}$ and $B^{(2)}$ being replaced by $T_3B^{(3)}$. This either gives a density increment for $A$ with respect to $T_3B'$, or else produces a set $\Delta$ satisfying most of the conditions of the final case of \cref{prop:incrementOrANS}. The rest of the proof is the same, after replacing every occurrence of $2\cdot A'$ by $-T_3A_3$ and every occurrence of $2\cdot B^{(i)}$ by $T_3B^{(i)}$.
\end{proof}

\Cref{prop:incrementOrSolutions} is the adaptation of \cite[Proposition~5.5]{BS} to general coefficients.

\begin{proposition}
\label{prop:incrementOrSolutions}
There is a constant $C>0$ such that, for all $k\geq C$, the following holds. Let $\B$ be a regular Bohr set of rank $d$ and suppose that $\A\subset \B$ has density $\alpha$. Let $\blue{\B^{\star} := \B \cap T_2\B \cap T_3\B}$. Either 
\begin{enumerate}
	\item $\alpha \geq 2^{-O(k^2)}$, 
	\item \begin{equation*}
	T(\A, \A, \A) \gg \exp\left(-\tilde{O}_{\alpha}(d \log{2d})\right) \mu(\blue{\B^{\star}})^2,
	\end{equation*}
	or
	\item $\A$ has a density increment of one of the following strengths \blue{relative to $\B_1:=\B^{\star}$, $\B_2:=T_2^{-1}\B^{\star}$ or $\B_3:=T_3^{-1}\B^{\star}$}: 
	\begin{enumerate}
		\item (small increment) $[\alpha^{O(\eps(k))}, \alpha^{-O(\eps(k))}; \tilde{O}_{\alpha}(1)]$, or
		\item (large increment) $[\alpha^{-1/k}, \alpha^{-1+1/k}; \tilde{O}_{\alpha}(1)]$,
	\end{enumerate}
	where $\eps(k) = \frac{ \log \log \log k}{\log \log k}$.
\end{enumerate}
\end{proposition}

\begin{proof}[Proof]
Let $\eps = c_0 \alpha^{C_0 \frac{\log\log \log k}{\log \log k}}$, for some small constant $0< c_0 \leq \frac13$ and some large constant $C_0 > 0$. We apply \cref{lem:bourgain} with
\begin{equation*}
B_1 = (\B^{\star})_{\rho}, \quad B_2 = T_2^{-1}(\B^{\star})_{\rho},\quad \text{and}\quad B_3 = T_3^{-1}(\B^{\star})_{\rho\rho'},
\end{equation*}
where $\rho = c\alpha\eps/d$ and $\rho' = c'\alpha^2/d$ ($c$ and $c'$ being small constants, chosen in particular such that $B_1$ and $B_3$ are regular.\footnote{Note that the regularity of $B_2$ follows immediately from that of $B_1$.} If we are in the second case of \cref{lem:bourgain}, then $\A$ has a small increment with respect to $B_1$, $B_2$ or $B_3$. By \cref{lem:relativeincrements}, this translates into a density increment of the same strength with respect to $\B_1$, $\B_2$ or $\B_3$, as required. 

Let us assume henceforth that we are in the first case of \cref{lem:bourgain}. Let
\begin{equation*}
A_1 = (\A - x) \cap B_1, \quad A_2 = (\A - x) \cap B_2\quad \text{and}\quad A_3 = (\A - x) \cap B_3.
\end{equation*}
If $\alpha_i$ is the density of $A_i$ relative to $B_i$, for $1\leq i\leq 3$, then \cref{lem:bourgain} ensures that 
\begin{equation*}
\alpha_i \in [(1-\eps)\alpha, (1+\eps)\alpha].
\end{equation*}

We now apply \cref{prop:incrementOrANS} with $B = B_1$, $B' = B_3$, $h = \lceil c_1 \log\log k/\log\log \log k\rceil$ and $t = \lceil C_2 \log k\rceil$, for some suitable constants $c_1, C_2 > 0$. 
\begin{enumerate}
	\item In the first case of the conclusion of \cref{prop:incrementOrANS}, $\alpha \gg 1/k^2 \geq 2^{-O(k^2)}$.
	\item In the second case, 
	\begin{equation*}
	T(A_1, A_2, A_3) \gg \alpha^3 \mu(B_1) \mu(B_3) \gg \exp\left(-\tilde{O}_{\alpha}(d\log 2d)\right) \mu(\B^{\star})^2
	\end{equation*}
	by \cref{lem:boundbohr}. Since the $A_i$'s are subsets of the same translate of $\A$ and the equation $a_1+T_2a_2+T_3a_3 = 0$ is translation-invariant, we have $T(\A, \A, \A) \geq T(A_1, A_2, A_3)$ and we are done.
	\item In the third case, either $A_1\subset B_1$ or $T_2A_2\subset B_1$ has a density increment of strength 
	\begin{equation*}
	[1, \alpha^{-1/k}; \tilde{O}_{\alpha}(h \log t)]\quad \text{or}\quad[\alpha^{-1/k}, \alpha^{-1+1/k}; \tilde{O}_{\alpha}(h \log t)]
	\end{equation*}
	with respect to $T_3B_3 = (B_1)_{\rho'}$. Note that $h \log t = \tilde{O}_{\alpha}(1)$, or else we have the first case of the conclusion. Therefore, $\A$ has a density increment of strength 
	\begin{equation*}
	[1, \alpha^{-1/k}; \tilde{O}_{\alpha}(1)]\quad \text{or}\quad[\alpha^{-1/k}, \alpha^{-1+1/k}; \tilde{O}_{\alpha}(1)]
	\end{equation*}
	relative to either $(B_1)_{\rho'}$ or $T_2^{-1}(B_1)_{\rho'} = (B_2)_{\rho'}$ (here we use the fact that $\eps$ is sufficiently small, similarly as in \cref{eq:epsilonchoice}). By \cref{lem:relativeincrements}, this implies that $\A$ has an increment of the same strength relative to $\B_1$ or $\B_2$.
	\item Finally, suppose that the last case of the conclusion of \cref{prop:incrementOrANS} holds. Then we may apply \cite[Proposition~11.8]{BS} with $B = B_1$, $B' = (T_3B_3)_{\rho_{\mathrm{top}}}$ and $B'' = (T_3B_3)_{\rho_{\mathrm{top}}\rho'}$. The hypotheses of \cite[Proposition~11.8]{BS} exactly match the last case of \cref{prop:incrementOrANS} for some $K = \alpha^{-O(1/k)}$.
	
	The number $M$ in \cite[Proposition~11.8]{BS} satisfies $M = \alpha^{-O(\eps(k))}$, or else $\alpha \geq 2^{-O(k^2)}$ and we are in the first case of our conclusion. Taking $C$ large enough in the statement of \cref{prop:incrementOrSolutions}, we see that the first case of \cite[Proposition~11.8]{BS} cannot hold. In the other two cases, either $A_1\subset B_1$ or $T_2A_2\subset B_1$ has a density increment of strength 
	\begin{equation*}
	[\alpha^{O(\eps(k))}, \alpha^{-O(\eps(k))}; \tilde{O}_{\alpha}(1)]\quad \text{or}\quad[\alpha^{-1/k}, \alpha^{-1+1/k}; \tilde{O}_{\alpha}(1)]
	\end{equation*}
	with respect to $B''$. As in the previous case, we conclude that $\A$ has a density increment of the same strength with respect to $\B_1$ or $\B_2$.\qedhere
\end{enumerate}
\end{proof}

We are now ready to prove \cref{thm:main}. The strategy is to iterate \cref{prop:incrementOrSolutions} as long as we are in the small increment case, and then apply \cref{thm:weakerbound} when one of the other cases applies. 

\begin{proof}[Proof of \cref{thm:main}]
Let $A\subset G$ of density $\alpha$. In this proof, $\alpha$ will always denote the density of this initial $A$.

Let $1\leq C_1 = O(1)$ be some absolute constant, chosen in particular larger than the implied constants in the exponents of the small increment case of \cref{prop:incrementOrSolutions}. Let $k$ be some constant large enough such that \cref{prop:incrementOrSolutions} holds and such that $8C_1\eps(k) \leq 1/2$. Let $1\leq C_2 = \tilde{O}_{\alpha}(1)$ be some fixed quantity (depending only on $\alpha$), chosen in particular larger than the implicit constants of \cref{prop:incrementOrSolutions} hidden in the $\gg$, $O(\cdot)$ and $\tilde{O}_{\alpha}(1)$ notation. By definition of $\tilde{O}_{\alpha}(\cdot)$, these implicit constants are still bounded by $C_2$ if we use \cref{prop:incrementOrSolutions} with some different relative density $\alpha'$, as long as $\alpha' \geq \alpha$. Note that we may assume that $2^{C_2k^2} \leq \alpha^{-1/2}$, or else we are done by an application of \cref{thm:weakerbound} with $B = G$.

\medbreak\textbf{Iterative construction.}
We inductively construct two sequences $(A_n)$ and $(B_n)$, where, for each $n$, $A_n$ is a subset of $B_n$ relative density $\alpha_n$. Let $A_0 = A$ and $B_0 = G$. Assume that $A_i$ and $B_i$ have been constructed for $i<n$. We use \cref{prop:incrementOrSolutions} with $\A = A_{n-1}$ and $\B = B_{n-1}$. If we are not in case (3)(a), then we stop the construction of the sequences. Otherwise, case (3)(a) occurs, and we have a small increment for $A_n$. Hence, there are sets $A_n \subset B_n$ such that
\begin{itemize}
	\item $A_{n}$ is a subset of a translate of $A_{n-1}$;
	\item $A_n$ is a subset of $B_n$ of relative density $\alpha_n \geq \left(1+C_2^{-1}\alpha^{C_1\eps(k)}\right)\alpha_{n-1}$;
	\item $B_n$ is a regular Bohr set of the form 
	\begin{equation}
		\label{eq:inters1}
		B_n = (S_n B_{n-1}^{\star})_{\rho_n} \cap \widetilde{B}_{n},
	\end{equation}
	where \begin{itemize}
		\item $B_{n-1}^{\star}$ is the Bohr set
		\begin{equation}
		\label{eq:inters2}
		B_{n-1}^{\star}:= B_{n-1} \cap T_2 B_{n-1}\cap T_3 B_{n-1},
		\end{equation}
		whose rank we denote by $d_{n-1}^{\star}$,
		\item $S_n$ is either $\mathrm{Id}_G$, $T_2^{-1}$ or $T_3^{-1}$,
		\item $\widetilde{B}_n = \mathrm{Bohr}(\widetilde{\Gamma}_n, \widetilde{\nu}_n)$, where $|\widetilde{\Gamma}_{n}| \leq C_2\alpha^{-C_1\eps(k)}$ and 
	\begin{equation}
	\label{eq:boundwidth}
	\bigg(\frac{\rho_n}{4}\bigg)^{d_{n-1}^{\star}}\prod_{\gamma \in \widetilde{\Gamma}_n} \frac{\widetilde{\nu}_n(\gamma)}{8} \geq  \exp \left(-\tilde{O}_{\alpha} \left(d_{n-1}^{\star}+\alpha^{-C_1\eps(k)}\right) \log\left(d_{n-1}^{\star}\right)\right).
	\end{equation}
	\end{itemize}
\end{itemize}

\medbreak\textbf{Analysis of the algorithm.} Note that, since $1\geq \alpha_n \geq \left(1+C_2^{-1}\alpha^{C_1\eps(k)}\right)^n\alpha$, this construction must terminate in
\begin{equation*}
l = \tilde{O}_{\alpha}\left(\alpha^{-C_1\eps(k)}\right)
\end{equation*}
steps. We then arrive at $A_l\subset B_l$ for which one of the cases (1), (2) and (3)(b) of \cref{prop:incrementOrSolutions} applies.

Let $W_i$ be the set of all automorphisms obtained by composing $i$ elements of $\{\mathrm{Id}_{G}, T_2, T_3, T_2^{-1}, T_3^{-1}\}$. Since $T_3 = -\mathrm{Id}_{G} - T_2$, these automorphisms commute, which implies that $|W_i| \leq (2i+1)^2$.

An immediate induction using \cref{eq:inters1,eq:inters2} shows that 
\begin{equation}
\label{eq:bohrsetsweaker}
B_i^{\star} \supset  \bigcap_{j=1}^{i} \bigcap_{T\in W_{2i+1}} (T \widetilde{B}_j)_{\rho_{j+1}\cdots \rho_i}
\end{equation}
for $0\leq i\leq l$. Similarly, we see that the frequency set of $B_i^{\star}$ is contained in 
\begin{equation*}
 \bigcup_{j=1}^{i} \bigcup_{T\in W_{2i+1}} T \widetilde{\Gamma}_j.
\end{equation*}
This implies that
\begin{equation*}
d_{l}^{\star} \ll l^3 C_2\alpha^{-C_1\eps(k)} = \tilde{O}_{\alpha}\left(\alpha^{-4C_1\eps(k)}\right).
\end{equation*}
In particular, \cref{eq:boundwidth} becomes
\begin{equation}
\label{eq:boundwidthweaker2}
	\bigg(\frac{\rho_n}{4}\bigg)^{d_{n-1}^{\star}}\prod_{\gamma \in \widetilde{\Gamma}_n} \frac{\widetilde{\nu}_n(\gamma)}{8} \geq   \exp \left( - \tilde{O}_{\alpha}\left(\alpha^{-4C_1\eps(k)}\right) \right),
\end{equation}
for $1\leq n \leq l$. 

We now use \cref{lem:boundbohr} to give a lower bound for $\mu(B_l^{\star})$. By \cref{eq:bohrsetsweaker}, we have 
\begin{equation*}
\mu(B_l^{\star}) \geq \prod_{j=1}^{l}\Bigg( \bigg(\frac{1}{4}\prod_{n=j+1}^{l} \rho_n\bigg)^{\mathrm{rk}(\widetilde{B}_j)}\prod_{\gamma\in \widetilde{\Gamma}_j} \frac{\widetilde{\nu}_j(\gamma)}{8}\Bigg)^{|W_{2l+1}|}.
\end{equation*}
Using \cref{eq:boundwidthweaker2} and the fact that $\mathrm{rk}(\widetilde{B}_j) \leq d_{n-1}^{\star}$ for  $j+1\leq n\leq l$, this yields
\begin{equation}
\label{eq:sizelastBohrset}
\mu(B_l^{\star}) \geq  \exp \left( - \tilde{O}_{\alpha}\left(\alpha^{-4C_1\eps(k)}\right) \right)^{ \tilde{O}_{\alpha}\left(\alpha^{-4C_1\eps(k)} \right)} \geq  \exp \left( - \tilde{O}_{\alpha}\left(\alpha^{-8C_1\eps(k)}\right) \right).
\end{equation}
If $B_l^{\star} = \mathrm{Bohr}(\Gamma_l^{\star}, \nu_l^{\star})$, this reasoning actually shows the more precise bound
\begin{equation}
\label{eq:sizelastBohrset2}
\prod_{\gamma\in \Gamma_l^{\star}} \frac{\nu_l^{\star}(\gamma)}{8} \geq \exp \left( - \tilde{O}_{\alpha}\left(\alpha^{-8C_1\eps(k)}\right) \right).
\end{equation}

\medbreak\textbf{Concluding the proof.}
We now apply \cref{prop:incrementOrSolutions} to $\A = A_l$ and $\B = B_l$. The small increment case cannot occur, by construction of the sequences $(A_l)$ and $(B_l)$.
\begin{itemize}
	\item If we are in the case (1) of \cref{prop:incrementOrSolutions}, then $\alpha_l \geq 2^{-C_2k^2}$. In this case we apply \cref{thm:weakerbound} and obtain the bound
	\begin{equation*}
	T(A_l, A_l, A_l) \gg \exp \left(-\tilde{O}_{\alpha}\left(\alpha^{-4C_1\eps(k)}+ 2^{C_2k^2}\right) \right) \bigg( \prod_{\gamma\in \Gamma_l} \frac{\nu_l(\gamma)}{8}\bigg)^{\tilde{O}_{\alpha}(1)},
	\end{equation*}
	where $B_l = \mathrm{Bohr}(\Gamma_l, \nu_l)$. Using \cref{eq:sizelastBohrset2}, we deduce that
	\begin{equation*}
	T(A, A, A) \geq T(A_l, A_l, A_l) \geq \exp \left( - \tilde{O}_{\alpha}\left(\alpha^{-8C_1\eps(k)}+ 2^{C_2k^2}\right) \right).
	\end{equation*}
	
	\item In the second case, we directly obtain 
	\begin{equation*}
	T(A_l, A_l, A_l) \gg \exp \left(-\widetilde{O}_{\alpha}\left(\alpha^{-4C_1\eps(k)}\right)\right) \mu(B_l)^2,
	\end{equation*}
	and thus 
	\begin{equation*}
	T(A, A, A) \geq T(A_l, A_l, A_l) \geq  \exp \left( - \tilde{O}_{\alpha}\left(\alpha^{-8C_1\eps(k)}\right) \right)
	\end{equation*}
	by \cref{eq:sizelastBohrset}.
	\item Finally, in the large increment case, there are some $\rho>0$, $T\in \{\mathrm{Id}_G, T_2^{-1}, T_3^{-1}\}$ and $\widetilde{B} = \mathrm{Bohr}(\widetilde{\Gamma}, \widetilde{\nu})$ such that ${|\widetilde{\Gamma}| \leq C_2\alpha^{-1+1/k}}$ and, if 
	\begin{equation*}
	B'' := (TB_l^{\star})_{\rho} \cap \widetilde{B},
	\end{equation*}
	then $\norm{\ind{A}*\m{B''}} \geqs{\alpha} \alpha^{1-1/k}$ and
	\begin{equation}
	\label{eq:sizeboundlargeinc}
	\left(\frac{\rho}{4}\right)^{\mathrm{rk}(TB_l^{\star})} \prod_{\gamma\in \widetilde{\Gamma}} \frac{\widetilde{\nu}(\gamma)}{8} \geq \exp \left(-\tilde{O}_{\alpha}\left(\alpha^{-4C_1\eps(k)}+\alpha^{-1+1/k}\right)\right).
	\end{equation}
	Write $B'' = \mathrm{Bohr}(\Gamma'', \nu'')$. Then
	\begin{align*}
	\prod_{\gamma\in \Gamma''} \frac{\nu''(\gamma)}{8} &\geq \left(\left(\frac{\rho}{4}\right)^{\mathrm{rk}(TB_l^{\star})} \prod_{\gamma\in \widetilde{\Gamma}} \frac{\widetilde{\nu}(\gamma)}{8}\right) \left(\prod_{\gamma\in \Gamma_l^{\star}} \frac{\nu_l^{\star}(\gamma)}{8}\right) \\
	&\geq \exp \left(-\tilde{O}_{\alpha}\left(\alpha^{-8C_1\eps(k)}+\alpha^{-1+1/k}\right)\right)
	\end{align*}	
	by \cref{eq:sizeboundlargeinc,eq:sizelastBohrset2}. We now apply \cref{thm:weakerbound} to a suitable subset $A''$ of a translate of $A$ and the Bohr set $B''$ to find that
	\begin{align*}
	T(A'', A'', A'') &\geq \exp \left(-\tilde{O}_{\alpha}\left(\alpha^{-4C_1\eps(k)}+\alpha^{-1+1/k}\right)\right) \bigg( \prod_{\gamma\in \Gamma''} \frac{\nu''(\gamma)}{8}\bigg)^{\tilde{O}_{\alpha}(1)} \\
	&\geq \exp \left(-\tilde{O}_{\alpha}\left(\alpha^{-8C_1\eps(k)}+\alpha^{-1+1/k}\right)\right).
	\end{align*}
\end{itemize}
Therefore, we obtain, in all three cases, the lower bound
\begin{equation*}
T(A, A, A) \geq\exp \left(-\tilde{O}_{\alpha}\left(\alpha^{-8C_1\eps(k)}+\alpha^{-1+1/k}+ 2^{C_2k^2}\right)\right).
\end{equation*}
Choosing $c = 1/(2k)$, say, we obtain 
\begin{equation*}
T(A, A, A) \geq \exp \left(-{O}(\alpha^{-1+c})\right).
\end{equation*}
On the other hand, since $A$ contains only trivial solutions to $a_1 + T_2a_2 + T_3a_3 = 0$, we have 
\begin{equation*}
T(A, A, A) = \frac{\alpha}{|G|} \leq \frac{1}{|G|}.
\end{equation*}
Therefore, $|G| \leq \exp\left({O}(\alpha^{-1+c})\right)$, which can be rewritten as
\begin{equation*}
|A| \ll \frac{|G|}{(\log{|G|})^{1+c'}},
\end{equation*}
where $c' = \frac{1}{1-c}-1$. This finishes the proof of \cref{thm:main}.
\end{proof}


\section*{Acknowledgments} 
I am deeply grateful to Thomas Bloom and Olof Sisask for bringing this problem to my attention and for the many fruitful conversations. I would also like to thank Timothy Gowers for his guidance and the very helpful discussions throughout my research internship for the \'Ecole Normale Supérieure.

\bibliographystyle{amsplain}
\providecommand{\bysame}{\leavevmode\hbox to3em{\hrulefill}\thinspace}
\providecommand{\MR}{\relax\ifhmode\unskip\space\fi MR }
\providecommand{\MRhref}[2]{%
  \href{http://www.ams.org/mathscinet-getitem?mr=#1}{#2}
}
\providecommand{\href}[2]{#2}


\begin{dajauthors}
\begin{authorinfo}[pgom]
  C\'edric Pilatte\\
  \'Ecole Normale Supérieure\\
  Paris, France\\
  cedric.pilatte\imageat{}ens\imagedot{}fr \\
\end{authorinfo}
\end{dajauthors}

\end{document}